\documentclass[reqno,a4paper,10pt]{amsart}
\usepackage{amsmath,amsfonts,amssymb,amscd,amsthm,amsbsy,upref,mathrsfs}
\usepackage[T1]{fontenc}

\newtheorem{theorem}{Theorem}[section]
\newtheorem{lemma}[theorem]{Lemma}
\newtheorem{proposition}[theorem]{Proposition}
\newtheorem{corollary}[theorem]{Corollary}

\newtheorem{fact}[theorem]{Fact}
\newtheorem*{claim}{Claim}
\newtheorem{definition}[theorem]{Definition}
\newtheorem*{notation}{Notation}

\theoremstyle{remark}
\newtheorem{remark}[theorem]{Remark}
\newtheorem{remarks}[theorem]{Remarks}

\newcommand{\bl}{\begin{lemma}}
\newcommand{\el}{\end{lemma}}
\newcommand{\bfa}{\begin{fact}}
\newcommand{\efa}{\end{fact}}
\newcommand{\bpr}{\begin{proposition}}
\newcommand{\epr}{\end{proposition}}
\newcommand{\bp}{\begin{proof}}
\newcommand{\ep}{\end{proof}}
\newcommand{\bd}{\begin{definition}}
\newcommand{\ed}{\end{definition}}
\newcommand{\bt}{\begin{theorem}}
\newcommand{\et}{\end{theorem}}
\newcommand{\bc}{\begin{corollary}}
\newcommand{\ec}{\end{corollary}}
\newcommand{\bn}{\begin{notation}}
\newcommand{\en}{\end{notation}}
\newcommand{\br}{\begin{remark}}
\newcommand{\er}{\end{remark}}
\newcommand{\bcl}{\begin{claim}}
\newcommand{\ecl}{\end{claim}}

\newcommand{\vvert}[1][\cdot]{\vert #1\vert}
\newcommand{\N}{{\mathbb{N}}}
\newcommand{\R}{{\mathbb{R}}}

\newcommand{\nrm}[1]{\|#1\|}
\newcommand{\al}{\alpha}
\newcommand{\e}{\varepsilon}

\newcommand{\bnum}{\begin{enumerate}}
\newcommand{\enum}{\end{enumerate}}
\newcommand{\mc}{\mathcal}
\newcommand{\mt}{\mc{T}}
\newcommand{\ms}{\mc{S}}

\newcommand{\fa}{f_{\alpha}}
\newcommand{\fb}{f_{\beta}}
\newcommand{\Q}{\mathbb{Q}}

\numberwithin{subsection}{section}
\numberwithin{equation}{section}

\newcommand{\norm}[1][\cdot]{\lVert #1\rVert}

\DeclareMathOperator{\supp}{supp}
\DeclareMathOperator{\maxsupp}{maxsupp}
\DeclareMathOperator{\minsupp}{minsupp}

\DeclareMathOperator{\suc}{succ}

\DeclareMathOperator{\ord}{ord}

\DeclareMathOperator{\conv}{conv}
\begin{document}
\title{Strictly singular non-compact operators\\ on a class of HI spaces}
\author{Antonis Manoussakis}
\address{Department of Sciences, Technical University of Crete, Greece}
\email{amanouss@aegean.gr}
\author{Anna Pelczar-Barwacz}
\address{Institute of Mathematics, Faculty of Mathematics and Computer Science, Jagiellonian University, {\L}ojasiewicza 6, 30-348 Krak\'ow, Poland}
\email{anna.pelczar@im.uj.edu.pl}
\begin{abstract}
We present a method for constructing bounded strictly singular non-compact operators on mixed Tsirelson spaces defined either by the families $(\mathcal{A}_n)$ or $(\mathcal{S}_n)$ of a certain class, as well as on spaces built on them, including hereditarily indecomposable spaces.
\end{abstract}
\maketitle
\section*{Introduction}
W.T. Gowers and B. Maurey, solving the unconditional problem in \cite{GM}, introduced the notion of an hereditarily indecomposable (HI) Banach space. Recall that a Banach space X is indecomposable if it cannot be written as a direct sum of its two infinitely dimensional closed subspaces, and is hereditarily indecomposable if every closed subspace is indecomposable.  The main feature of a complex HI Banach  space, shared also by Gowers-Maurey space, is that every bounded operator on the space is a strictly singular perturbation of a multiple of the identity. S.A. Argyros and A.  Tolias, \cite{Atol} proved that in any HI Banach space every bounded operator is either strictly singular or has finitely dimensional kernel and restricted to the complement of its kernel is an isomorphism. The results on the structure of the operators on HI spaces  brought the attention  to the "scalar-plus-compact" problem stated by J. Lindenstrauss \cite{Lin}, who asked if there exists a Banach space on which every bounded operator is a compact perturbation of a multiple of the identity. The prime candidate to be considered in this context was Gowers-Maurey space $X_{GM}$. However, after W.T. Gowers, \cite{gow2} who built a bounded strictly singular non-compact operator on a subspace of  $X_{GM}$, G. Androulakis and  Th. Schlumprecht, \cite{as} constructed a strictly singular non-compact operator  on the  whole space $X_{GM}$. I. Gasparis \cite{gasp} achieved the same in Argyros-Deliyanni space $X_{AD}$ of \cite{ad2}. Finally, the  "scalar-plus-compact" problem  was solved recently by S.A. Argyros and R. Haydon, \cite{AH}, who gave a method to construct an $\mathscr{L}_\infty$ HI Banach spaces with the  "scalar-plus-compact" property. The existence of  a reflexive  Banach space with the "scalar-plus-compact" property remains an open problem. In this context a natural question concerns the existence of a  non-trivial strictly singular operators on mixed Tsirelson spaces and HI spaces who have these spaces as  a frame. Recall here that non-trivial strictly singular operators were constructed also on subspaces of mixed Tsirelson spaces and spaces built on them \cite{aost, kmp, p}. 

In the present paper we present a method for constructing bounded strictly singular non-compact operators in a class of mixed Tsirelson spaces defined either by the families $\mathcal{A}_{n}$ or Schreier families $\mathcal{S}_{n}$, as well as in spaces built on them, including HI spaces. Let us a recall that a mixed Tsirelson space  defined by a sequence of regular families  $\mathcal{F}_{n}$ of finite subsets of $\N$ and a sequence $\theta_{n}\searrow 0$, denoted by $T[(\mc{F}_{n},\theta_{n})_{n}]$, has norm satisfying  the implicit equation 
$$ 
\norm[x]=\max\left\{\norm[x]_{\infty}, \sup_{n}\theta_{n}\sup\left\{\sum_{i=1}^{d}\norm[E_{i}x]: (E_{i})_{i=1}^{d}-\mc{F}_{n}-\text{admissible}\right\}\right\}\,.
$$
The famous Schlumprecht space is the space $T[(\mc{A}_{n},1/\log_{2}(n+1))_{n}]$. Families $(\mc{S}_{n})$ define asymptotic $\ell_1$ mixed Tsirelson spaces introduced in \cite{ad2}. 

We construct a bounded strictly singular non-compact operator on any mixed Tsirelson space defined by the families $\mc{S}_{n}$, with the property  $\liminf_{n}\theta_{n+k}/\theta_{n}\geq c>0$ for any $k\in\N$, as well as in a class of Banach spaces built on its base, including the space $X_{AD}$ and their alike. Let us recall that I. Gasparis construction of an operator on the HI spaces  built on such mixed Tsirelson spaces  is based on the existence of $c_{0}^{\omega}$-spreading models in the dual space. In the spaces we consider such a property does not seem to hold in general. Gasparis method was adapted in \cite{ADT} for constructing HI spaces, built on spaces $T[(\mc{A}_{n_j},\frac{1}{m_j})_{j})] $, with diagonal strictly singular non-compact operators, as well as for constructing non-trivial strictly singular operators on asymptotic $\ell_p$ HI spaces \cite{B}.

Our method was inspired by the approach of \cite{as} and \cite{ADT}. As in case of \cite{as}, we base our construction on the unconditional components building spaces of the considered class (recall that in I. Gasparis construction strict singularity of the operator follows by HI property). In particular we consider spaces built on mixed Tsirelson spaces with the following property: the basis of the considered space is asymptotically equivalent to the basis of the underlying mixed Tsirelson space. 

In order to construct an operator on the space defined by Schreier families $(\mc{S}_n)$ we need a periodic version of rapidly increasing sequences. Using this notion we construct block sequences vectors $(x_n)$ and associated norming functionals $(f_n)$ with tree-analysis modeled on a fixed infinite tree. We show next that the  operator $Tx=\sum_{n=1}^{\infty}f_{q_n}(x)e_{t_n}$, for appropriate choice of  $(q_n)$, $(t_n)$, is bounded and strictly singular. This method applies also to a class of  spaces defined by families $(\mc{A}_n)$ with Schlumprecht space parameters property, i.e. satisfying  $\theta_{n}n^{a}\to +\infty$ for every $\alpha>0$. In such spaces our argument simplifies since we can use the classic notion of rapidly increasing sequences instead of its periodic versions, and our definition coincides with the construction of \cite{as}, however, we approach differently the constructed functionals in deriving the desired properties of the operator built on them.

The crucial feature of the tree construction of $(f_n)$ is that for appropriate choice of parameters it resembles property of being a $c_{0}^{\omega}$- or $c_{0}$-spreading model. More precisely we pick some $(r_j), (n_j)\subset\N$ with $(r_j/n_j)$ increasing arbitrary fast, such that for any $j< F\in\mc{S}_{r_j}$, the norm of $\sum_{n\in F}f_n$  is controlled, up to some "error" part, which should be dealt with separately, by $n_j$. This makes the behaviour of $(f_n)$ "close enough'' to $c_0$-type behaviour to obtain the desired properties of the operator. Dealing with the "error" part relies on asymptotic representation of the basis of underlying mixed Tsirelson space in the considered space, mentioned above.

Non-compactness of the operator is ensured by seminormalization of $(f_n)$, implied by seminormalization of corresponding vectors $(x_n)$. The last property, due to the periodic structure inscribed in the tree-analysis, requires delicate computation, made in original mixed Tsirelson spaces, and then transfered to a class of Banach spaces built on them. 

The paper is organized as follows. The first section contains the preliminaries. In the second section we consider a general method of construction of a bounded strictly singular non-compact operator on a given space, whose variant we shall use in the case of spaces built on mixed Tsirelson spaces. In the third section we recall basic properties of the mixed Tsirelson spaces defined both by the  families $(\mc{A}_n)$ and $(\mc{S}_n)$. The forth and fifth sections contains the main results. In the forth section we define the periodic averages and prove upper estimations for their norms, which provides tool for showing non-compactness of the constructed operator. The sixth section is devoted to the construction of a bounded strictly singular and non-compact operator on a space defined by the families  $\mc{S}_{n}$. At the end we sketch the construction in the case of the families  $\mc{A}_{n}$, simplifying the arguments used for the Schreier families.
\newline\textit{Acknowledgements.}
We are grateful to Professor Spiros Argyros for valuable discussions. 

\section{Preliminaries}

We recall  the basic definitions and standard notation.

By a {\em  tree}  we shall mean a non-empty partially ordered  set $(\mt, \leq)$ for which the set $\{ y \in \mt:y \preceq x \}$ is linearly ordered and finite for each $x \in \mt$. 
The tree $\mt$ is called {\em finite}  if the set $\mt$ is finite. The \textit{root} is the smallest element of the tree (if it exists). The {\em terminal}  nodes are the maximal elements. A {\em branch} of $\mt$ is any maximal linearly ordered set in $\mt$. The {\em immediate successors}  of $x \in \mt$, denoted by $\suc (x)$, are all  the nodes $y \in \mt$ such that $x \preceq y$ but there is no $z \in \mt$ with $x \preceq z \preceq y$. If $\mt$ has a root, then for any node $\al\in\mt$ we define a level of $\al$, denoted by $|\al|$, as the length of the branch linking $\al$ and the root. We refer to \cite{z2} for more information on trees. 

By $[\N]^{<\infty}$ we denote the family of all finite subsets of $\N$. 
Given any $\upsilon=(n_1,\dots, n_k)\in [\N]^{<\infty}$ and $n\in\N$ we let $\upsilon^\frown n=(n_1,\dots,n_k,n)$.  A family of  $\mc{F}$ of finite subsets of $\N$ is said to be regular if it is 
\begin{enumerate}
\item compact as a subset of  $\{0,1\}^{\N}$ with the product topology,
\item hereditary, i.e.  if $F\subset G$ and $G\in\mc{F}$ then $F\in \mc{F}$
\item spreading, i.e. if $m_{1}\leq n_{1},\dots, m_{d}\leq n_{d}$ then $(n_{i})_{i=1}^{d}\in\mc{F}$ whenever $(m_{i})_{i=1}^{d}\in\mc{F}$.
\end{enumerate}
Given a family $\mc{F}\subset [\N]^{<\infty}$ we say that the sequence of sets $E_1,\dots, E_d\subset \N$ is $\mc{F}$-admissible, if $\max E_1< \min E_2\leq \max E_2<\dots<\min E_d$ and $(\min E_i)_{i=1}^d\in \mc{F}$.  Given two families $\mc{F}, \mc{G}\subset [\N]^{<\infty}$ we define $\mc{F}[\mc{G}]$ as
$$
\mc{F}[\mc{G}]=\{F_1\cup\dots\cup F_d:\   F_1,\dots, F_d\in \mc{G}, \ (F_1, \dots, F_d)-\mc{F}\text{-admissible} \}\,.
$$
We work on two types of families of finite subsets of $\N$: $(\mc{A}_n)_{n\in\N}$ and $(\mc{S}_n)_{n\in\N}$. Let $\mc{A}_n=\{F\subset\N:\# F\leq n\}$, $n\in\N$, and define inductively \textit{Schreier families} $(\mc{S}_n)_{n\in\N}$, \cite{aa}, by
\begin{align*}
\mc{S}_0 &=\{\{ k\}:\ k\in\N\}\cup\{\emptyset\}, \\
\mc{S}_{n+1}&  =\mc{S}_1[\mc{S}_n], \ \ n\in\N\,.
\end{align*}
It it well known that the families  $(\mc{A}_{n})$, $(\mc{S}_{n})$ are regular.

Let $X$ be a  Banach space with basis $(e_i)$. The \textit{support} of a vector $x=\sum_{i} x_i e_i$ is the set $\supp x =\{ i\in \N : x_i\neq 0\}$. Given any $x=\sum_{i} a_ie_i$ and finite $E\subset\N$ put $Ex=\sum_{i\in E}a_ie_i$. We write $x<y$ for vectors $x,y\in X$, if $\max\supp x<\min \supp y$. A \textit{block sequence} is any sequence $(x_i)\subset X$ satisfying $x_{1}<x_{2}<\dots$, a \textit{block subspace} of $X$ - any closed subspace spanned by an infinite block sequence. 

Let $\mc{F}$ be a regular family of finite subsets  of $\N$. We say that a finite block sequence $(x_{i})_{i=1}^{d}$ is $\mc{F}$-admissible if  $(\minsupp(x_{i}))_{i=1}^{d}\in\mc{F}$. Given a family $\mc{F}$ and a scalar $\theta\in (0,1]$ by the $(\mc{F}, \theta)$-operation we mean an operation which associates with any $\mc{F}$-admissible sequence $(x_1,\dots,x_d)$ the average $\theta(x_1+\dots+x_d)$.

We shall use in the sequel the following
\br\label{ps} Let $X$ be a Banach space with basis and $G=(x_{n}^{*})_n$ be a seminormalized block sequence  in the dual space.  Let also $\mc{F}$ be a compact family of finite subsets of $\N$. Then the space $X_G$ defined as the completion of  $c_{00}$ under the norm
$$
\norm[x]_{G}=\sup\left\{\sum_{n\in F}\e_{n}x_{n}^{*}(x): \e_{n}\in\{-1,1\},\ F\in\mc{F}\right\}, \ \ \ x\in c_{00}
$$
is $c_{0}$-saturated, by the result of A. Pe\l{}czy\'{n}ski and Z. Semadeni \cite{PS}.
\er
\section{A non-compact strictly singular operator in a general situation}
We present now a general result is inspired by \cite{ADT}. It represents a general method, whose variant we shall use in case of spaces built on mixed Tsirelson spaces. 
\bpr\label{adt}
Let $X$ be a Banach space with a basis $(e_n)_n$ such that there exists a sequence of regular families $(\mc{F}_{n})_{n}\subset [\N]^{<\infty}$ with $\mc{F}_k\subset\mc{F}_n$ for any $k\leq n$, and a norming set  $D$ for $X$ such that
\bnum
\item for any $\e>0$ there is an $n_{\e}\in\N$ with the following property: for any $f\in D$ we have 
$$
\{n\in\N:\ |f(e_n)|>\e\}\in\mc{F}_{n_{\e}}
$$
\item there is a basic seminormalized sequence $(x_n^*)\subset X^*$, a strictly increasing $(N_j)\subset\N$ and $(\e_j)\subset (0,1)$ such that $\sum_j\e_jN_j<\infty$ and for any $j\in\N$, any $F\in\mc{F}_{n_{\e_{j+1}}}$  with  $F\geq n_{\e_{j+1}}$ we have $\norm[\sum_{n\in F}x_n^*]\leq N_j$. 
\enum
Then the operator $T:X\to X$ given by the formula $T(x)=\sum_{n=1}^\infty x_{q_n}^*(x)e_n$, where $q_{j}=n_{\e_{j+1}}$ for each $n\in\N$, is bounded. If additionally we assume that $X$ does not contain a copy of $c_0$, then $T$ is strictly singular.
\epr
The proof  of boundedness of $T$ is just repetition of the proof of Prop. 3.1 \cite{ADT} with necessary adjustments. We present this proof here as it will serve as a basis for showing also strict singularity of $T$. 
\bp 
Let $x\in X$ and $f\in D$. For any $j\in\N$ we let
\begin{align*}
B_j&=\{ n\in\N:\ \e_{j+1}<|f(e_n)|\leq \e_j\}\\
C_j^1&=\{n\in B_j:\ n\geq j,\ x_{q_n}^*(x)\geq 0,\ f(e_n)\geq 0\}\\
C_j^2&=\{n\in B_j:\ n\geq j,\ x_{q_n}^*(x)< 0,\ f(e_n)\geq 0\}\\
C_j^3&=\{n\in B_j:\ n\geq j,\ x_{q_n}^*(x)\geq 0,\ f(e_n)<0\}\\
C_j^4&=\{n\in B_j:\ n\geq j,\ x_{q_n}^*(x)< 0,\  f(e_n)<0\}\\
D_j&=B_j\cap\{1,\dots,j\}
\end{align*}
Then the set $A=\{q_n:\ n\in C_j^k\}$ satisfies $q_j\leq A\in\mc{F}_{q_j}$ by (1), and by (2) satisfies $\norm[\sum_{n\in C_j^k}x_{q_n}^*]\leq N_j$ for any $j\in\N$, $k=1,2,3,4$.

In order to show boundedness we compute
\begin{align*}
|f(\sum_{n\in\N}x_{q_n}^*(x)e_n)|&\leq \sum_{j=1}^\infty|\sum_{n\in B_j}x_{q_n}^*(x)f(e_n)|\\
&\leq \sum_{j=1}^\infty|\sum_{n\in D_j}x_{q_n}^*(x)f(e_n)|+\sum_{j=1}^\infty\sum_{k=1}^4|\sum_{n\in C_j^k}x_{q_n}^*(x)f(e_n)|\\
& \leq \sum_{j=1}^\infty \e_j j\norm[x]+\sum_{j=1}^\infty\sum_{k=1}^4|\sum_{n\in C_j^k}x_{q_n}^*(x)|\max_{n\in C_j^k}|f(e_n)|\\
& \leq \sum_{j=1}^\infty (\e_j j+4\e_j N_j)\norm[x]
\end{align*}
Now we assume additionally that $X$ does not contain a copy of $c_0$. Consider the norm $\norm_G$, with $G=\{x_{q_n}^*: \ n\in\N\}$ and $\mc{F}=\mc{F}_{q_{j_0}}$ defined as in Remark \ref{ps}. 
As $\norm_{G}\leq (N_{j_0}+q_{j_0})\norm$ it follows that for any $\e>0$ in any block subspace of $X$ there is a vector $x$ with $\norm[x]=1$ and $\norm[x]_{G}<\e$. Fix $j_0\in\N$ and take  a vector $x$  with $\norm[x]_G<(j_0)^{-2}$ and $\norm[x]=1$. Take $f\in D$, define the sets $(B_j)$ as above and estimate
\begin{align*}
|f(\sum_{n\in\N}x_{q_n}^*(x)e_n)|&\leq \sum_{j=1}^\infty|\sum_{n\in B_j}x_{q_n}^*(x)f(e_n)|\,.
\end{align*}
For the part $\sum_{j=j_0}^\infty(|\sum_{n\in B_j}x_{q_n}^*(x)f(e_n)|)$ we repeat the previous estimations, obtaining the upper bound $5\sum_{j=j_0}^\infty\e_j n_j$. By the choice of $x$ we have also
\begin{align}\label{g-norm}
\sum_{j=1}^{j_0-1}|\sum_{n\in B_j}x_{q_n}^*(x)f(e_{t_n})| &\leq \sum_{j=1}^{j_0-1} \sum_{n\in B_j}|x_{q_n}^*(x)|\,|f(e_{t_n})|
\\
&=\sum_{j=1}^{j_0-1}\rho_{r_{j}}\sum_{n\in B_j}\e_{n}x_{q_n}^*(x)\,\,\text{for suitable $\e_{n}\in\{-1,1\}$} \notag
 \\
 &
 \leq j_{0} \norm[x]_G\leq \frac{1}{j_0} . \notag
\end{align}
Putting the estimations altogether we obtain $\norm[Tx]\leq 5\sum_{j=j_0}^{\infty}\e_jn_j+1/j_0$ while $\norm[x]=1$. Since we can pick such vector $x$ for any $j_0$ and in any subspace of $X$, we finish the proof of strict singularity of $T$. 
\ep
In case of spaces built on mixed Tsirelson spaces the large part of the work will be devoted to showing that the chosen sequence of functionals is seminormalized (Section 4), and estimations on the norms of the sum of functionals will be more delicate than the assumption (2) (Section 5). 
\section{Spaces  $T[(\mc{F}_{n},\theta_{n})_{n}]$ - basic facts}
We recall now the definition of mixed Tsirelson spaces and their basic properties.

\bd[Mixed Tsirelson space] Fix a sequence of families $(\mc{F}_n)=(\mc{A}_{k_n})$ or $(\mc{S}_{k_n})$ and sequence $(\theta_n)\subset (0,1)$ with $\lim_{n\to\infty}\theta_n=0$. Let $K\subset c_{00}$ be the smallest set satisfying the following:
\bnum 
 \item $(\pm e_n^*)_n\subset K$,
 \item $K$ is closed under the $(\mc{F}_n,\theta_n)$-operation on block sequences for any $n$. 
\enum 
We define a norm on $c_{00}$ by $\nrm{x}=\sup\{f(x):f\in K\}$, $x\in c_{00}$. The \textit{mixed Tsirelson space} $T[(\mc{F}_n,\theta_n)_n]$ is the completion of $(c_{00}, \nrm{\cdot})$.
\ed
It is standard to verify that the norm $\nrm{\cdot}$  is the unique norm on $c_{00}$ satisfying the equation
$$
\nrm{x}=\max\left\{\nrm{x}_\infty,\sup\left\{\theta_n\sum_{i=1}^k\nrm{E_ix}: \ (E_i)_{i=1}^k - \mc{F}_n- \text{admissible}, \ n\in\N\right\}\right\}\,.
$$
It follows immediately that the u.v.b. $(e_n)$ is 1-unconditional in the space $T[(\mc{F}_n,\theta_n)_n]$. It was proved in \cite{ad2} that any $T[(\mc{S}_{k_n},\theta_n)_n]$ is reflexive, also any $T[(\mc{A}_{k_n},\theta_n)_n]$ is reflexive, provided $\theta_n>\frac{1}{k_n}$ for at least one $n\in\N$, \cite{ato}. 

Recall that any space $T[(\mc{A}_{k_n},\theta_n)_n]$ is isometric to some space $T[(\mc{A}_n,\hat{\theta}_n)_n]$ satisfying the following regularity condition: $\hat{\theta}_{n}\hat{\theta}_{m}\leq\hat{\theta}_{nm}$, $n,m\in\N$, therefore we assume in the sequel that the considered space is of this regular form. In this case we shall assume additionally that  $\theta_{n}n^{a}\to +\infty$ for every $a>0$ (i.e. the considered space is a 1-space in the terminology of \cite{m}).    

Analogously any space $T[(\mc{S}_{k_n},\theta_n)_n]$ is isometric to some space $T[(\mc{S}_n,\hat{\theta}_n)_n]$ satisfying the following regularity condition: $\hat{\theta}_{n}\hat{\theta}_{m}\leq\hat{\theta}_{n+m}$, $n,m\in\N$, and thus we assume again that any considered space defined by Schreier families is of this regular form.  

In the sequel we shall  state the results and present proofs valid in both cases of spaces defined by $(\mc{A}_n)$ and  $(\mc{S}_n)$, formulating them in terms of tree-analysis of norming functionals and suitable averages. In case of $(\mc{A}_n)$ the reasoning gets much more simplified, as we shall point out. 

From now on $X$ denotes the space $T[(\mc{F}_n,\theta_n)_n]$, where $(\mc{F}_n)_n=(\mc{A}_n)_n$ or $(\mc{S}_n)_n$  with $(\theta_n)$ satisfying the above regularity conditions, and $K$ denotes its norming set.

The following notion provides a useful tool for estimating norms in mixed Tsirelson spaces:
\bd[The tree-analysis of a norming functional]\label{def-tree} Let $f\in K$. By a \textit{tree-analysis} of $f$ we mean a finite family $(f_\al)_{\al\in \mt}$ indexed by a tree $\mt$ with a root $\emptyset\in \mt$ (the smallest element) such that the following hold
\bnum
 \item $f_\emptyset=f$ and $f_\al\in K$ for all $\al\in \mt$,
 \item $\al\in T$ is maximal if and only if $f_\al\in (\pm e_n^*)$,
 \item for every not maximal $\al\in T$ there is some $n\in\N$ such that $(f_\beta)_{\beta\in\suc (\al)}$ is an $\mc{F}_n$-admissible  sequence and $f_\al=\theta_n(\sum_{\beta\in\suc (\al)}f_\beta)$. We call $\theta_n$ the \textit{weight} of $f_\al$ and denote by $w(f_\al)$. 
 \enum
 \ed
Notice that every functional $f\in K$ admits a tree-analysis, not necessarily unique.
\begin{notation}
Given any norming functional $f\in K$ with a tree-analysis $(f_\al)_{\al\in\mt}$ and a block sequence $(x_i)$ we say that $f_\al$ covers $x_{i}$, if $\al$ is maximal in $\mt$ with $\supp f_\al\supset\supp x_{i}\cap \supp f$.
\end{notation}
We shall use also the following standard facts.
\begin{fact}\label{fact1} For any block sequence of norming functionals $(f_i)$ on $X$ and any $x\in c_{00}$ we have $\norm[\sum_i f_i(x)e_i]\leq \norm[x]$.
\end{fact}
This Fact  is an immediate consequence of the fact that $\norm[\sum_{i}a_{i}e_{i}]\leq \norm[\sum_{i}a_{i}u_{i}]$ for every normalized block sequence $(u_i)$ and any $(a_{i})\subset\R$. 

We shall use  the notion of a skipped set. Recall that a set $L\subset \N$ is $M$-skipped, for $M\in\N$, if for any $n,m\in L$ we have $|n-m|\geq M$. The following lemma is an easy consequence of the definition of an $M$-skipped set.
\begin{lemma} \label{l1}
Let $(a_{j})_{j\in \N}$ be a decreasing sequence   of positive numbers with $\sum_{j\in\N}a_j=1$  and $L\in [\N]$  be an  $M$-skipped subset of $\N$, $M\in\N$. Then
\begin{equation}\label{eq0}
\sum_{j\in L}a_{j}\leq  \frac{1}{M}+\sup_{j\in L}a_{j}.
\end{equation}
\end{lemma}
In our proofs we shall use the notion of the basic average and the basic special convex combination and some of its basic properties which we recall below.
\begin{definition}[\cite{S}, \cite{ADKM}, Special averages]$ $ Fix $n\in\N$ and $\e>0$.

 A) We say that a vector  $x=\frac{1}{n}\sum_{i=1}^{n}e_{k_{i}}$ is an $\e$-$\ell_1^n$-vector if $k_{1}<\dots<k_{n}$ and $1/n<\e$.

B)  We say that a vector $x=\sum_{i\in F}a_ie_i$ is an $(n,\e)$-basic s.c.c. (special convex combination), 
if $F\in\mc{S}_n$ and scalars $(a_i)\subset [0,1]$ satisfy $\sum_{i\in F}a_i=1$ and $\sum_{i\in G}a_i<\e$ for any $G\in\mc{S}_{n-1}$. 
\end{definition}
In the sequel in order to unify the notation we shall call a vector $x\in T[(\mc{F}_n,\theta_n)_n]$ an $(n,\e)$-\textit{special basic average} either if it is an $\e-\ell_1^n$-vector in case $\mc{F}_{n}=\mc{A}_{n}$, $n\in\N$,   or it is an $(n,\e)$-basic s.c.c. in case $\mc{F}_{n}=\mc{S}_{n}$, $n\in\N$. We say that a vector $x=\sum_{i\in F}a_ix_i$ is an $(n,\e)$-special average  of a normalized block sequence $(x_i)_{i\in F}$, if the vector $\sum_{i\in F}a_ie_{s_{i}}$ is an $(n,\e)$-basic special average, 
$s_{i}=\maxsupp x_{i}$. Recall that for any $n,m\in\N$ and $\e>0$ there is an $(n,\e)$-basic special average $x$ with $x>m$.

In this unified terminology we state the next two lemmas, which contain all the properties of special averages we shall need in the sequel.
\begin{lemma}\cite{S}, \cite{ADKM}
  \label{normbasicscc}
Let  $x=\sum_{i\in F}a_{i}e_{i}\in X$ be an  $(n,\e)$-basic special average. Then
$$
\theta_{n}\leq \norm[x]\leq \theta_{n}+\e\,.
$$
\end{lemma}
\begin{lemma}\cite{S}, \cite{ADKM}
  \label{lem1}
For every $\e>0$ and $k\in\N$ there exist  $w(\e, k)\in\N$ such that  for every $n\geq w(\e,k)$, for every $(n, \e)$-basic special average $x=\sum_{i\in F}a_{i}e_{i}\in X$ and every $f\in K$ with $w(f)\geq\theta_{k}$ the following holds
\begin{equation}
  \label{eq01}
 \vert f(\theta_{n}^{-1}x)\vert\leq (1+\e)w(f).
\end{equation}
\end{lemma}
\section{Periodic averages and their norm}
This section is devoted to the vectors of a special type in mixed Tsirelson spaces. Their role in our construction is to ensure the seminormalization of the associated norming functionals building the operator, which in turn provide non-compactness of the operator. The final result of this section is stated in Lemma \ref{??}. 
\begin{definition}\label{repdef}
Fix $n_{0}\in\N, \e,\tilde{\e}>0$, a block sequence $x_{1},\dots,x_{NM}$ with $c=\max_{p=1,\dots,NM}\norm[x_p]<\infty$, and parameters $(n_i)_{i=0}^{M+1}, (q_i)_{i=0}^{M+1}$,   satisfying the following
\begin{enumerate}
 \item[(P1)] 
$q_{0}<n_{0}<q_{1}<n_{1}<\dots< q_{M+1}$, 
$\theta_{q_{1}}\leq\tilde{\e}\theta_{n_{0}}^{2}/2,$ 
$n_{0}\geq w(\e,q_{0})$ 
and $\theta_{q_{i+1}}\leq\theta_{n_{i}}^2$ for all $i$,
\item[(P2)] for every $f\in K$ with $w(f)\geq \theta_{q_j}$ we have
\begin{equation}
\qquad \label{es1}
\vvert[f(x_{(k-1)M+i} ) ]\leq c(1+\tilde{\e})w(f)\,\,\,\text{for every $k$ and every}\,\,\ i=j,\dots,M.\notag
\end{equation}
\item[(P3)] for every $f\in K$ with $w(f)\leq \theta_{q_{j+1}}$ we have
\begin{equation}
\label{es2}
\qquad \vvert[f(x_{(k-1)M+i})]\leq c(1+\tilde{\e})\theta_{n_{i}}\,\,\text{ for every $k$ and every}\,\,\ i=1,\dots,j.
\notag
\end{equation}
\end{enumerate}
Let  $x=\sum_{k=1}^{N}\sum_{j=1}^{M}a_{(k-1)M+j}x_{(k-1)M+j}$  be an $(n_{0},\e)$-special average.  We shall call the sequence $\{x_{(k-1)M+j}: k\leq N, j\leq  M\}$ an $(n_{0},\e,\tilde{\e},M)$-periodic RIS (rapidly increasing sequence) and we shall call $x$ also an $(n_{0},\e,\tilde{\e},M)$-periodic average. 
\end{definition}
\begin{remarks}
1) Taking  $N=1$ we obtain the standard notion of RIS of special types of vectors, originated in \cite{GM} and  \cite{ad2}.  

2)  In the case of families $(\mc{A}_n)$ we have $NM=n_{0}$. 

3) In order to construct the operator with desired properties on spaces $T[(\mc{A}_n,\theta_n)_n]$ and spaces based on them it is enough to consider only the simplest version of the above notion, i.e. with $N=1$ and $M=n_0$, as it was used in \cite{as}. The full strength of the above notion, making use of the periodic repetition of weights in the sequence $(x_i)$ will be necessary in the spaces defined by Schreier families.  
\end{remarks}
M-skipped sets as well as  a more complicated form of periodic averages have been used by D. Leung and W.K. Tang, \cite{LTang}, to provide new examples of mixed Tsirelson spaces  $T[(\mc{S}_{n},\theta_n)_n]$ not isomorphic to their modified version.

The next proposition gives us upper bounds of the norm of a periodic average. 
We use in the sequel the following notation: for a fixed block sequence $(x_{n})_{n}$ we let 
$s_n=\maxsupp(x_n)$ for any $n$.
\begin{proposition}\label{main}
Let  $x=\sum_{k=1}^{N}\sum_{j=1}^{M}a_{(k-1)M+j}x_{(k-1)M+j}$  be an $(n_{0},\e,\tilde{\e}, M)$-periodic average of $x_1,\dots,x_{NM}$, with $N,M\in\N$, $\theta_{n_0}\leq 1/10$, $0<\e,\tilde{\e}\leq \theta_{n_{0}}^{2}$,  and $\norm[x_p]\leq 1$ for each $p=1,\dots,NM$. Then for every $f\in K$ there exists $g\in conv(K)$ such that
\begin{equation}
  \label{eq:9}
 \vvert[f(x)]
\leq
(1+\tilde{\e})g(\sum_{k=1}^{N}\sum_{j=1}^{M}a_{(k-1)M+j}e_{s_{(k-1)M+j}})+\e_{1}+\e_{2}
\end{equation}
where $w(f)=w(g)$ if   $w(f)\not\in [\theta_{q_{M}+1},\theta_{q_{1}}]$  otherwise either $w(f)=w(g)$ or $g=\lambda e^{*}_{p}+(1-\lambda)\theta_{n_{0}}h$, $h\in K$, for some $\lambda\in[0,1]$, $\e_{1}=(1+\tilde{\e})\theta_{n_{1}}$ and $\e_{2}=2(\frac{1}{M}+\sup_{p}a_{p})(1-\theta_{1})^{-1}$.
In particular 
$$
\norm[\theta_{n_{0}}^{-1}x]\leq (1+\e)(1+\tilde{\e})+(\e_{1}+\e_{2})\theta_{n_{0}}^{-1}.
$$
Moreover if  
\begin{equation}\label{condM}
4(\frac{1}{M}+\sup_{p}a_{p})\leq (1-\theta_{1})\theta_{n_{0}}^{3}
\end{equation}
then  
$\norm[\theta_{n_{0}}^{-1}x]\leq (1+2\theta_{n_{0}})$ and 
 \begin{equation}  \label{eq:10}
 \vvert[f(\theta_{n_{0}}^{-1}x)]
\leq 
\begin{cases}
 (1+3\theta_{n_{0}})w(f)\,\,\,&\text{if}\,\,w(f)\geq\theta_{q_{0}}
\\
 (1+3\theta_{n_{0}})\theta_{n_{0}}\,\,\,&\text{if}\,\,w(f)\leq\theta_{q_{1}}
\end{cases}
\end{equation}
\end{proposition}
\begin{proof}
Take the parameters $(n_{i}), (q_{i})$ defining $x$ as in (P1) and let $f$ be a norming functional with the tree-analysis $(\fa)_{\alpha\in\mc{T}}$. For every $\alpha\in\mc{T}$ we set
$$
C_{\alpha}=\{p: \fa\,\text{covers}\,\,\ x_{p}\}\ \text{ and }\ \widetilde{C}_{\alpha}=\{p: \supp x_{p}\subset\supp \fa\}
$$
Then $\widetilde{C}_{\alpha}=C_{\alpha}\cup\bigcup_{\beta\in\suc\alpha}\widetilde{C}_{\beta}$.
Let
$$
I_{k}=[(k-1)M+1,\dots, kM]\,\,\,\text{for}\,\, k=1,\dots,N.
$$
For every $\alpha\in\mc{T}$ we set
$$
U_{\al}=\{p\in C_{\al}: p=(k-1)M+j\,\text{for some  $k,j$ and}\,\,\, w(\fa)\in (\theta_{q_{j+1}},\theta_{q_{j}} ] \}.
$$
Note that there are no  $j\ne j^{\prime}$ such that  $(k-1)M+j,(k^{\prime}-1)M+j^{\prime}\in U_{\al}$.

We set  $R_{\alpha}=U_{\alpha}$ if $\# U_{\alpha}\geq 2$, otherwise  $R_{\alpha}=\emptyset$.

Let   $n<h(\mc{T})$. Set $K_{n}=\cup_{\vert\alpha\vert=n} R_{\alpha}$. It follows easily   that  the set $K_{n}$ is the union of at most two  $M$-skipped sets. It follows from \eqref{eq0} that for every $n$,
\begin{equation}
  \label{eq:3}
  \sum_{p\in K_{n}}a_{p}\leq  2(\frac{1}{M}+\sup_{p}a_{p})=\e_2(1-\theta_1).
\end{equation}
Inductively for every $\alpha\in\mt$ we  define  a   functional $g_{\alpha}\in \conv(K) $ such that
\begin{equation}
  \label{eq:4}
  \vvert[\fa (\sum_{p\in \widetilde{C}_{\alpha}}a_{p}x_{p})]\leq 
(1+\tilde{\e})g_{\alpha}(\sum_{p\in   \widetilde{C}_{\alpha}}a_{p}e_{s_{p}})
+\e_{\alpha}^{1}+\e_{\alpha}^{2}
\end{equation}
where $w(f_\al)=w(g_\al)$ if   $w(f_\al)\not\in [\theta_{q_{M}+1},\theta_{q_{0}}]$  otherwise
either $w(f_\al)=w(g_\al)$ or $g_\al=\lambda e^{*}_{s_{p}}+(1-\lambda)\theta_{n_{0}}h$, $h\in K$ for some
$\lambda\in[0,1]$,
and
\begin{align}\label{errors1}\e_{\alpha}^{1}=(1+\tilde{\e})\theta_{n_{1}}
\sum_{p\in\widetilde{C}_{\alpha}}a_{p}\,\,\text{ and}\,\,\,\,
\e_{\alpha}^{2}=\sum_{p\in  R_{\al}}a_{p}+w(\fa)\sum_{\beta\in\suc(\al)}\e_{\beta}^{2}.
\end{align}
From \eqref{eq:3}, \eqref{errors1} it follows that
\begin{equation}
  \label{eq:6}
 \e^{2}_\emptyset\leq \e_2(1-\theta_1)\sum_{n\geq 0}\theta_{1}^{n}\leq 
\e_2.
\end{equation}
If $\widetilde{C}_{\alpha}=\emptyset$ we set $g_{\alpha}=0$.

Assume that $\widetilde{C}_{\alpha}\ne\emptyset$.    If $C_{\alpha}=\emptyset$ we set $g_{\alpha}=w(\fa)\sum_{\beta\in\suc(\al)}g_{\beta}$, $\e_{\al}^{1}, \e_{\al}^{2}$ as in \eqref{errors1}.

Assume that  $C_{\alpha}\ne\emptyset$. We distinguish the following cases, examining the cardinality of    $U_{\al}$.

\textit{CASE 1}.   $U_{\alpha}=\emptyset$.

In this case  we set  
$$
g_{\alpha}=w(\fa)(\sum_{p\in C_{\alpha}}e^{*}_{s_{p}}+\sum_{\beta\in succ(\alpha)}g_{\beta})
$$
If $(k-1)M+j\in C_{\alpha}$ and  $w(f)\geq \theta_{q_{j}}$, (P2) yields
 \begin{align}\label{eqkm-j}
 f_{\alpha} ( a_{(k-1)M+j}x_{ (k-1)M+j}) &\leq (1+\tilde{\e})a_{(k-1)M+j}w(\fa) 
\\
&\leq (1+\tilde{\e})g_{\alpha}(a_{(k-1)M+j}e_{s_{(k-1)M+j}})
\notag
 \end{align}
If $(k-1)M+j\in C_{\alpha}$ and  $w(\fa)\leq \theta_{q_{j+1}}$, (P3) yields
 \begin{align}\label{eqkm+j}
 f_{\alpha} ( a_{(k-1)M+j}x_{ (k-1)M+j}) &\leq
 (1+\tilde{\e})\theta_{n_j}a_{(k-1)M+j}\leq (1+\tilde{\e})\theta_{n_1}a_{(k-1)M+j}
\\
&\leq g_{\alpha}(a_{(k-1)M+j}e_{s_{(k-1)M+j}})+(1+\tilde{\e})\theta_{n_1}a_{(k-1)M+j}.
\notag
 \end{align}
For any $\beta\in\suc(\alpha)$ we have
 \begin{align}\label{ecov}
 f_{\alpha}(&\sum_{p\in \widetilde{C}_{\beta}}a_{p}x_{p})=w(f_\al)f_{\beta}( \sum_{p\in \widetilde{C}_{\beta}}a_{p}x_{p})
 \\
 & \leq 
 (1+\tilde{\e})w(f_\al)g_{\beta}( \sum_{p\in \widetilde{C}_{\beta}}a_{p}e_{s_{p}})+w(f_\al)\e_{\beta}^{1}+w(f_\al)\e_{\beta}^{2}\quad 
&\text{by inductive hypothesis}\notag
 \\
&=(1+\tilde{\e})g_{\al}( \sum_{p\in \widetilde{C}_{\beta}}a_{p}e_{s_{p}})+w(f_\al)(\e_{\beta}^{1}+\e_{\beta}^{2}) &
\notag
 \end{align} 
From \eqref{eqkm-j}, \eqref{eqkm+j} and \eqref{ecov} we get that \eqref{eq:4} holds, as
\begin{align*}
 \fa(\sum_{p\in \widetilde{C}_\al}a_px_p)\leq (1+\tilde{\e})g_\al(\sum_{p\in \widetilde{C}_\al}a_pe_{s_p})+w(\fa)\sum_{\beta\in\suc(\al)}\e^2_\beta+(1+\tilde{\e})\theta_{n_1}\sum_{p\in\widetilde{C}_\al}a_p
\end{align*}

\textit{CASE 2}. $\# U_{\alpha}=1$.

Let $k_{0},j_{0}$ be such that $p_{0}=(k_{0}-1)M+j_{0}\in U_{\alpha}$ and  $\theta_{q_{j_0}+1}< w(f_\al)\leq \theta_{q_{j_{0}}}$. We set
\begin{align}\label{1aa}
g_{\alpha}=\frac{1}{1+\tilde{\e}}e^{*}_{s_{(k_{0}-1)M+j_{0}}}+
\frac{\tilde{\e}}{1+\tilde{\e}}\theta_{n_{0}}\tilde{g}_{\alpha}
\end{align}
where 
$\tilde{g}_{\alpha}=\theta_{n_0}(\sum_{p_{0}\ne p \in C_{\alpha}}e^{*}_{s_{p}}+\sum_{\beta\in\suc(\al)}g_{\beta})$.

Since $x$ is an $\mc{F}_{n_{0}}$-special average, the set  $\{\fb: \widetilde{C}_{\beta}\ne\emptyset\}\cup \{s_{p}: p\in C_{\alpha}\}$ is at most $\mc{F}_{n_0}$-admissible and hence 
$$
\tilde{g}_{\alpha}=\theta_{n_{0}}(\sum_{p_{0}\ne p\in C_{\alpha}}e_{s_{p}}^{*}+\sum_{\beta\in\suc(\alpha)}g_{\beta})\in \conv(K).
$$
We have
\begin{equation}\label{eqj0}
\vvert[\fa(x_{(k_{0}-1)M+j_{0}})]\leq 1.
\end{equation}
Note that if $(k-1)M+j\in  C_{\alpha}$ for $ j>j_{0}$,  (P2) yields
\begin{align}
\label{eqj0+j}\vert \fa(x_{(k-1)M+j})\vert &\leq
(1+\tilde{\e})w(f_\al)\leq
(1+\tilde{\e})\frac{\tilde{\e}}{1+\tilde{\e}}\theta^{2}_{n_{0}},\,\,\text{by (P1) } 
\\
&=(1+\tilde{\e})\frac{\tilde{\e}}{1+\tilde{\e}}\theta_{n_{0}}\tilde{g}_{\al}(e_{s_{p}})=(1+\tilde{\e})g_{\al}(e_{s_{p}})
\notag
\end{align}
Also  (P3) and (P1) yield that for every $(k-1)M+j\in C_{\al}$, $j<j_{0}$
\begin{equation}\label{eqj0-j}
\vvert[\fa(x_{(k-1)M+j})]\leq (1+\tilde{\e})\theta_{n_{j}}\leq
(1+\tilde{\e})\frac{\tilde{\e}}{1+\tilde{\e}}\theta^{2}_{n_{0}}.
\end{equation}
For any $\beta\in\suc(\alpha)$, as in \eqref{ecov}, we have, using $w(f_\al)\leq\theta_{q_1}<\tilde{\e}\theta^{2}_{n_{0}}/2$, 
 \begin{align}\label{ecov2}
 f_{\alpha}(&\sum_{p\in \widetilde{C}_{\beta}}a_{p}x_{p})\leq (1+\tilde{\e})g_{\al}( \sum_{p\in \widetilde{C}_{\beta}}a_{p}e_{s_{p}})+w(f_\al)(\e_{\beta}^{1}+\e_{\beta}^{2}) &
 \end{align}
Therefore  by  \eqref{eqj0}, \eqref{eqj0+j}, \eqref{eqj0-j} and \eqref{ecov2} we get
  \begin{align*}
  f_{\alpha}(\sum_{p\in\widetilde{C}_{\alpha}}a_px_p) &\leq 
  a_{p_{0}}+(1+\tilde{\e})g_{\al}(\sum_{p_{0}\ne p\in \widetilde{C}_{\al}}a_{p}e_{s_{p}})+
\\
&\qquad\quad +(1+\tilde{\e})\theta_{n_{1}}\sum_{p\in C_{\al}}a_{p}+w(\fa)\sum_{\beta\in\suc(\al)} \e_{\beta}^{1}+w(f_\al)\sum_{\beta\in\suc(\al)}\e_{\beta}^{2}
  \\
  &=\frac{1}{1+\tilde{\e}} (1+\tilde{\e})a_{p_{0}}e^{*}_{p_{0}}(e_{p_{0}})+ (1+\tilde{\e})g_{\al}(\sum_{p_{0}\ne p\in \widetilde{C}_{\al}}a_{p}e_{s_{p}}) +\e_{\al}^{1}+\e_{\al}^{2}
  \\
  & = (1+\tilde{\e}) g_{\alpha} ( \sum_{p\in\widetilde{C}_\al}a_{p}e_{s_p})+\e_{\al}^{1}+\e_{\al}^{2}.
  \end{align*}
 It is clear  that  \eqref{eq:4} holds.
\\

\textit{CASE 3}.  $\# U_{\alpha}\geq 2$

From the definition of $U_{\alpha}$ we get that there exists $j_{0}\in \{1,\dots,M\}$ such that $(k-1)M+j_{0}\in U_{\alpha}$ for at least two $k$ and $w(f_\al)\in (\theta_{q_{j_0}+1},\theta_{q_{j_0}}]$.

Note that $\bigcup_{k}\{(k-1)M+j_{0}\in U_{\alpha}\}$ is an $M$-skipped sequence. Let 
$$
g_{\alpha}=w(f_\al)(\sum_{p\in C_{\alpha}}e^{*}_{s_p}+\sum_{\beta\in\suc(\alpha)}g_{\beta})
$$
Repeating the reasoning from Case 1 we obtain
\begin{align*}
 \fa(\sum_{p\in \widetilde{C}_\al\setminus R_\al}a_px_p)\leq
 (1+\tilde{\e})g_\al(\sum_{p\in \widetilde{C}_\al\setminus R_\al}a_pe_{s_p})+w(\fa)\sum_{\beta\in\suc(\al)}(\e^{1}_{\beta}+\e^2_\beta) +(1+\tilde{\e})\theta_{n_1}\sum_{p\in\widetilde{C}_\al}a_p
\end{align*}
We have also 
\begin{align*}
 \fa(\sum_{p\in R_\al}a_px_p)\leq \sum_{p\in R_\al}a_p
\end{align*}
which proves  \eqref{eq:4}.

For the moreover part, for $f\in K$ with $w(f)\geq \theta_{q_{0}}$  \eqref{eq:9} yields
\begin{align}\label{fes2}
  \vvert[f(\theta_{n_{0}}^{-1}x)]&\leq
 (1+\tilde{\e})g_{f}(\theta_{n_{0}}^{-1}\sum_{p}a_{p}e_{s_{p}})+(\e_{1}+\e_{2})\theta_{n_{0}}^{-1}
  \\
&\leq (1+\tilde{\e})(1+\e)w(f)+\theta_{n_{0}}^{2}/2+\theta^{2}_{n_{0}}/2\notag
  \\
  &\leq  (1+3\theta_{n_{0}})w(f)\,\,\,\textrm{by the choice $\tilde{\e},\e<\theta_{n_{0}}^{2}$}.
  \notag
\end{align}
Also for every $f\in K$ with $w(f)\leq\theta_{q_1}$, \eqref{eq:9} yields
\begin{align}\label{fes3}
  \vvert[f(\theta^{-1}_{n_{0}}x)]
  &\leq (1+\tilde{\e})g_{f}(\theta_{n_{0}}^{-1}\sum_{p}a_{p}e_{s_{p}})+(\e_{1}+\e_{2})\theta_{n_{0}}^{-1}
    \\
 &\leq (1+\tilde{\e})\max\{\theta_{n_{0}}^{-1}w(f), \theta_{n_{0}}^{-1}\sup_{p}a_{p}+\tilde{\e}\theta_{n_{0}}\norm[\theta_{n_{0}}^{-1}x]\} +\theta^{2}_{n_{0}} 
\notag
   \\
 &\leq (1+3\theta_{n_{0}})\theta_{n_{0}}  
\notag
\end{align}
\end{proof}
\begin{remark}\label{remmain}
1) When considering the simplest version of periodic averages, i.e. with $N=1$ (the case we shall use in the sequel for mixed Tsirelson spaces defined by families $(\mc{A}_n)$) the proof above reduces to the Cases 1 and 2, as $\# U_\al\leq 1$ always, with $\e_2=0$. In case of spaces $T[(\mc{A}_n, \theta_n)_n]$ with $\theta_nn^a\to\infty$ for any $a>0$ the result follows by \cite{S}.

2) If in case of $T[(\mc{A}_n,\theta_n)_n]$ we take $(n_{0},\e,\tilde{\e},M)$-periodic RIS with $N\geq 1$, then using that $NM=n_{0}$  and  $\theta_{n}n^{a}\to +\infty$ for every $a>0$ we can always choose arbitrary large $M$ and $n_0$ to satisfy inequality \eqref{condM} and the "moreover" part to hold.
\end{remark}

\br\label{main-col}  The above result enables us to construct repeated periodic averages. 

The first inductive step of the construction is provided by the following observation. Fix $n_0,M\in\N$ with $\theta_{n_0}\leq 1/10$ and $0<\e,\tilde{\e}\leq \theta_{n_{0}}^{2}$. Let $(x_{(k-1)M+j})_{k=1,\dots,N,j=1,\dots,M}\subset X$ be a block sequence such that each $x_{(k-1)M+j}$ is an $(n_j,\e_j)$-basic special average, with parameters satisfying (P1) and $n_j\geq w(\e,q_j)$ for each $j$. Then by Lemma \ref{lem1} the sequence $(\theta_{n_j}^{-1}x_{(k-1)M+j})_{ k=1,\dots,N,j=1,\dots,M}$ satisfies also (P2) and (P3) and thus we can define an $(n_0,\e,\tilde{\e},M)$-periodic average of this sequence. 

The general inductive step of the construction is given by the following remark. Fix $n_0,M\in\N$ with $\theta_{n_0}\leq 1/10$ and $0<\e,\tilde{\e}\leq \theta_{n_{0}}^{2}$. Consider a block sequence $(x_{(k-1)M+s})_{k=1,\dots,N, s=1,\dots,M}\subset X$ such that each $x_{(k-1)M+s}$ is an $(n_{s},\e_s, \tilde{\e}_s,M_s)$-periodic average, with parameters $M_s$, $n_s$, $(a_{p,s})_p$ satisfying \eqref{condM} and $\e_s,\tilde{\e}_s<\theta_{n_s}^2/2$. Assume also that for some sequence $(q_s)$ we have  
$$
n_s\geq w(\e_s,q_s), \ \ \theta_{q_{s+1}}\leq \tilde{\e}_s\theta_{n_s}^2/2 \text{ for each }s
$$
and that first two parameters of the sequence $(q_{i,s})_{i=0}^{M_s+1}$ defining each $x_{(k-1)M+s}$ is given by $q_{0,s}=q_s$ and $q_{1,s}=q_{s+1}$. Then by Prop. \ref{main} applied to the sequence $((1+3\theta_{n_1})^{-1}\theta_{n_s}^{-1}x_p)_{p=1}^{NM}$ the sequence $(\theta_{n_{s}}^{-1}x_{(k-1)M+s})_{k=1,\dots,N,s=1,\dots,M}$ satisfies also (P2) and (P3) and thus we can define an $(n_0,\e,\tilde{\e},M)$-periodic average of this sequence.
\er
With the above we are ready to define repeated averages based on the periodic structure described above. 
\bd A core tree is any tree $\mc{R}\subset\cup_n\N^n$ on $\N$ with tree order $\preceq$ and lexicographic order $\leq_{lex}$ with no terminal nodes and a root $\emptyset$ with associated parameters $(m_\mu)_{\mu\in\mc{R}}\subset\N$. For any $\mu\in\mc{R}$ we let also $M_\mu=\#\suc (\mu)$. 
\ed
Such a tree served as a model for tree-analysis in \cite{as} in $(\mc{A}_n)$ setting. 
\begin{notation}
We enumerate nodes of $\mc{R}$ according to the lexicographic order as $(\mu_j)_{j\in\N}$, starting with $\mu_0=\emptyset$. For simplicity we shall write also $m_j=m_{\mu_j}$ and $M_j=M_{\mu_j}$ for any $j\in\N$.

Define also $\ord(\mu)$ for $\mu\in\mc{R}$ as $\ord(\emptyset)=0$ and $\ord(\mu)=\sum_{\gamma\in\mc{R}:\ \gamma\prec\mu}m_\gamma$ for any $\mu\neq\emptyset$.
\end{notation}
\begin{definition}[Periodic RIS tree-analysis]\label{RIS-tree-an}
We say that a vector $x\in X$ has a periodic RIS tree-analysis of height $n\in\N$ with core $\mc{R}$, if there is a family $(x_\al)_{\al\in\mt}$, indexed by a tree $\mt$ on $\N$ of height $n$ with a root and a mapping 
$\upsilon:\mt\to\mc{R}$ with the following properties:
\begin{enumerate}
\item for any terminal node $\al\in\mt$ we have $|\al|=n$ and $x_\al=e_{t_\al}$ for some $t_\al\in \N$,
\item for any node $\al\in\mt$ with $|\al|=n-1$ the vector $x_\al$ is a "seminormalized" $(m_{\upsilon(\al)},\e_\al)$-basic special average of $(x_\beta)_{\beta\in\suc(\al)}$ i.e. $x_{\al}=\theta_{m_{\upsilon(\al)}}^{-1}\sum_{i\in F_{\alpha}}c_{i}e_{i}$ for some $F_\al\in\ms_{m_{\upsilon(\al)}}$ and $\e_{\al}=\theta_{m_{\upsilon(\al)}}^2$.
\item for any node $\al\in\mt$ with $|\al|<n-1$ the vector $x_\al$ is a "seminormalized" $(m_{\upsilon(\al)},\e_\al, \tilde{\e}_\al, M_{\upsilon(\al)})$-periodic average of $(x_\beta)_{\beta\in\suc(\al)}$, where 
$$\suc(\al)=\{\al^\frown((k-1)M_{\upsilon(\al)}+j):\ k=1,\dots, N_\al, \ j=1,\dots, M_{\upsilon(\al)}\}\,,$$ 
that is 
$$
x_\al=\theta^{-1}_{m_{\upsilon(\al)}}\sum_{k=1}^{N_\al}\sum_{j=1}^{M_{\upsilon(\al)}}a_{\al^\frown((k-1)M_{\upsilon(\al)}+j)}x_{\al^\frown((k-1)M_{\upsilon(\al)}+j)}
$$
with $\e_\al=\theta_{m_{\upsilon(\al)}}^3/4$, $\tilde{\e}_\al= \theta_{m_{\upsilon(\al)}}^2$.
\item $\upsilon(\emptyset)=\emptyset$ and for any $\al\in\mt$ with $|\al|<n-1$ 
we have 
$$
\upsilon (\al^\frown((k-1)M_{\upsilon(\al)}+j))=\upsilon(\al)^\frown j, \ \ \ k=1,\dots, N_\al, \ j=1,\dots, M_{\upsilon(\al)},
$$
\end{enumerate}
\end{definition}
\br\label{rem}

1) If we consider the simplest version of the above notion, setting  $N_\al=1$ for every node $\al\in\mt$ with $|\al|<n-1$, we obtain the notion of "repeated averages" of \cite{amt}.

2) Notice that in the case of families $(\mc{A}_n)$, if we consider the simplest version of periodic averages, i.e. $N_\al=1$ for each $\al$ with $|\al|<n-1$, then $M_{\upsilon(\al)}=m_{\upsilon(\al)}$ and thus we can identify, via the mapping $\upsilon$, the tree $\mt$ with restriction $\mc{R}\cap\cup_{i=0}^n\N^i$ of the tree $\mc{R}$. Here our construction coincides with the approach of \cite{as}.

3) For any $\mu\in\mc{R}$ the family $(x_\al)_{\upsilon(\al)=\mu}$ is $\mc{F}_{\ord(\mu)}$-admissible. 

4) Notice that $\supp x\in\mc{S}_{p_n}$, for some $p_n$ depending only on the core tree $\mc{R}$ and the height $n$.

\er
Given a vector  $x$ with a periodic RIS tree-analysis $(x_{\al})_{\al\in \mt}$ as in Def. \ref{RIS-tree-an} we define in a natural way \textit{an associated norming functional} 
 $f\in K$ with a tree-analysis $(\fa)_{\al\in\mt}$ as follows (with the notation of Def. \ref{RIS-tree-an}):
\begin{enumerate}
\item for any terminal node $\al\in\mt$ we set $f_\al=e_{t_\al}^*$,
\item for any node $\al\in\mt$ with $|\al|=n-1$ we set  
$f_{\al}=\theta_{m_{\upsilon(\al)}}\sum_{i\in F_{\alpha}}e_{i}^*$,
\item for any node $\al\in\mt$ with $|\al|<n-1$ we set 
$$
f_\al=\theta_{m_{\upsilon(\al)}}\sum_{k=1}^{N_\al}
\sum_{j=1}^{M_{\upsilon(\al)}}
f_{\al^\frown((k-1)M_{\upsilon(\al)}+j)}
$$
\end{enumerate}

For any $\mu\in\mc{R}$  let $c_\mu=\prod_{\gamma\prec\mu, \gamma\neq\mu}\theta_{m_\gamma}$. For simplicity, with abuse of notation we shall write $c_{\mu_j}=c_j$ for any $j\in\N$.

Notice that for any $\al\in\mt$ with $\upsilon(\al)=\mu$ we have $f|_{\supp f_\al}=c_\mu f_\al$.
\br\label{rem-z0} The key-property we achieve building repeated averages on the periodic RIS sequences in $(\mc{S}_n)$ setting is the following: whereas the number of descendants of each node in $\al\in\mt$ can be arbitrary large, depending on $\minsupp x$, we keep the number of weights of descendants of each node $\al\in\mt$ depending only on the parameter $M_\mu$ of the node $\mu\in\mc{R}$  with $\upsilon(\al)=\mu$. Therefore having fixed at the beginning a core tree $\mc{R}$ we control the number of weights of any level of the  tree-analysis of any norming functional with a core $\mc{R}$. 
\er
\begin{lemma}\label{??} With the above notation assume the following conditions:
\begin{enumerate}
\item[(R0)] $m_{j+1}\geq w(\theta_{m_j}^3/4,q_j)$ with some $q_j\in\N$ satisfying $\theta_{q_j}\leq \theta_{m_j}^4$ for any $j\in\N$,
 \item[(R1)] 
$\theta_{m_j}<1/2^{j+1}$ for any $j\in\N$,
\item[(R2)] $M_j>4(1-\theta_{1})^{-1}\theta_{m_j}^{-4}$ for any $j\in\N$.
\end{enumerate}
Then for any $n$ there exists a vector $x\in X$ with periodic RIS tree-analysis modeled on a core tree $\mc{R}$ with parameters $(m_j)$ and $(M_j)$ of height $n$. Moreover,  for every $\alpha\in\mt$ with $\vvert[\al]<n$ it holds
\begin{equation}
  \label{normxan}\norm[x_{\al}]\leq \prod_{j:\ |\mu_j|\geq\vvert[\al]}(1+3\theta_{m_{j}})
\end{equation}
In particular $\norm[x]\leq 3$ and $1/3\leq\norm[f]\leq 1$, where $f$ is the norming functional associated to $x$. 
\end{lemma}
\begin{proof}
Existence follow by Remark \ref{main-col}. We shall prove the "Moreover" part inductively on the level of $\al\in\mt$. For  $\vvert[\al]=n-1$ we have that $x_{\al}=\theta_{m_{\upsilon(\al)}}^{-1}\sum_{i\in F_{\al}}a_{i}e_{i}$ is a seminormalized $(m_{\upsilon(\al)}, \e_{\al})$-basic special average with  $\e_{\al}\leq\theta_{m_{\upsilon(\al)}}^2$, thus by Lemma~\ref{normbasicscc} we finish the proof. 

Assume that  \eqref{normxan} holds for every $\beta\in\mt_{n}$ with $0<\vvert[\beta]=k\leq n-1$. Let $\al\in\mt_{n}$ with $\vvert[\al]=k-1$. Then we have that
$$
x_{\al}=\theta_{m_{\upsilon(\al)}}^{-1}\sum_{k=1}^{N_{\al}}\sum_{j=1}^{M_{\upsilon(\al)}}
a_{\al^{\frown}(k-1)M_{\upsilon(\al)}+j}x_{\al^{\frown}(k-1)M_{\upsilon(\al)}+j}
$$
The assumption on $(M_j)$ and condition on $\e_\al^n$ in Def.~\ref{RIS-tree-an} guarantee that the "moreover" part of Prop.~\ref{main} holds true, therefore
\begin{align*}
\norm[x_{\al}]&\leq (1+3\theta_{m_{\upsilon(\al)}})\prod_{j:\ \vvert[\mu_j]\geq\vvert[\al]+1}(1+3\theta_{m_j})
\\
&\leq\prod_{j:\ \vvert[\mu_j]\geq\vvert[\al]}(1+3\theta_{m_j}).
\end{align*}
In particular by (R1) we get $\norm[x]\leq 3$ and, as $f(x)=1$, $1/3\leq \norm[f]\leq 1$. 
\end{proof}

\br Similar conditions were considered in \cite{as} as a tool to prove seminormalization of the sequence of norming functionals defining the non-trivial strictly singular operator on Schlumprecht and Gowers-Maurey space.
\er

\section{Construction of the operator}
We recall now the general idea of constructing Banach spaces of various properties, as HI spaces, on the base of mixed Tsirelson spaces, started by W.T.Gowers and B.Maurey in $(\mc{A}_n)$ setting \cite{GM} and S.Argyros I.Deliyanni in $(\mc{S}_n)$ setting \cite{ad2}. 

Let $X=T[(\mc{F}_n,\theta_n)_n]$ be a regular mixed Tsirelson space, with $(\mc{F}_n)=(\mc{A}_n)$ or $(\mc{S}_n)$. We fix also sequence $(\rho_l)_{l\in L}\searrow 0$, with infinite $L\subset\N$, such that $\rho_l\geq \theta_l$ for any $l\in L$.

Let $\mc{W}=\{(f_1,\dots,f_k):\ f_1<\dots<f_k\in c_{00}(\Q), \norm[f_i]_\infty\leq 1,\ i=1,\dots, k, \ k\in\N\} $ and fix an injective function $\sigma:\mc{W}\to N$.
For any $D\subset c_{00}(\Q)$ define
\begin{align*}
D_n=&\left\{\theta_n\sum_{i=1}^kf_i: \ f_1,\dots,f_k\in D,  \  (f_1,\dots,f_k) \ \mc{F}_n\text{-admissible}, \ k\in\N\right\},  \ \  n\in\N \\
D_l^\sigma=&\left\{\rho_l\sum_{i=1}^kEf_i: \ f_1,\dots,f_k\in D, \  (f_1,\dots,f_k) \ (\sigma,\mc{F}_l)\text{-admissible}, \  E\subset\N \text{ interval}, \ k\in\N\right\},   \ \  l\in L
\end{align*}
where a block sequence $(f_1,\dots,f_k)$ is $(\sigma, \mc{F}_l)$-admissible, if $(f_1,\dots,f_k)$ is $\mc{F}_l$-admissible, $f_1\in\bigcup_{n\in\N}D_n$ and $f_{i+1}\in D_{\sigma(f_1,\dots,f_i)}$ for any $i<k$.

Consider a symmetric set $D\subset c_{00}(\Q)$ such that
\bnum
 \item $(\pm e_n^*)_n\subset D$,
 \item $D\subset\bigcup_{n\in N}D_n\cup\bigcup_{l\in L}D_l^\sigma$,
 \item $D_n\subset D$ for any $n\in N$.
\enum
Define a norm on $c_{00}$ as $\norm[x]_D=\sup\{f(x):f\in D\}$, $x\in c_{00}$, denote by $X_D$ the completion of $(c_{00},\norm[\cdot]_D)$. Obviously the u.v.b. $(e_n)$ is a basis for $X_D$. Moreover $X_D$ is reflexive in case of Schreier families, see for example \cite{ato}, as well as in case of families $(\mc{A}_n)$, provided  $\theta_n>\frac{1}{n}$  for at least one $n\in\N$.

\

The family of spaces described above include the classical mixed Tsirelson spaces $T[(\mc{F}_n,\theta_n)_n]$ (taking $L=\N$ and $\rho_n=\theta_n$ for example) as well as Gowers-Maurey space and asymptotic $\ell_1$ HI spaces introduced by S.Argyros and I.Deliyanni in \cite{ad2}.

The following lemma describes the only properties of $X_D$ we shall  need in the sequel. In fact our construction of the non-trivial strictly singular operator hold true for any space $Z$ not containing $c_0$ with the properties (D1), (D2), (D3) stated below.

\bl The space $X_D$ satisfies the following
\bnum
\item[(D1)] the dual unit ball in $X_D^*$ is closed under $(\mc{F}_n,\theta_n)$-operations on block sequences,
\item[(D2)] the u.v.b. of $X_D$ is $(\mc{F}_n)$-equivalent to the u.v.b. of $X$, i.e. there is $C\geq 1$ and an increasing sequence $(i_n)\subset\N$ such that for any $n$ and $i_n\leq F\in\mc{F}_n$ the sequences $(e_i)_{i\in F}$ in $X$ and $({e}_i)_{i\in F}$ in $X_D$ are $C$-equivalent,
\item[(D3)] for any $f\in X_{D}^{*}$ with $\norm[f]_D\leq 1$ and any $l\in L$ we have $\{n: |f(e_n)|\geq\rho_l\}\in\mc{F}_l$,
\enum
\el
Notice that in case of families $(\mc{A}_n)$ condition (D2) states that the spreading model of the u.v.b. of $X_D$ is $C$-equivalent to the u.v.b. of $X$.
\bp (D1) follows from (3) in definition of $D$. The property (D2) in case of Schreier families is proved in \cite{p} (proof of Cor. 4.7), the same reasoning works in $(\mc{A}_n)$ case, in case of Gowers-Maurey space (D2) was proved in \cite{as}. (D3) follows immediately from (2) in definition of $D$.
\ep

In order to construct a non-trivial strictly singular operator on $X_D$ under appriopriate conditions on $(\theta_n)$ we shall define first a core tree $\mc{R}$ with parameters $(m_j)$, $(M_j)$ satisfying conditions (R0), (R1), (R2) and additional ones stated below. Having this we take the block sequences $(x_n)$ ,$(f_{n})$ where  each $x_{n}$  is defined by a periodic RIS tree-analysis with core $\mc{R}$ of height $n$, and $f_n$ is associated to $x_n$. By Lemma~\ref{??} and properties of $X_D$ the functionals $(f_{n})$ are seminormalized and we show that for appropriate sequences $(r_{n}), (t_n)$ the operator $T$ defined by $T(x)=\sum_{n}f_{r_n}(x)e_{t_n}$ is a bounded strictly singular non-compact operator on $X_D$.

\

We consider first the case of spaces defined by Schreier families. Fix a sequence $(\theta_n)\searrow 0$ with $\theta_n\theta_k\leq\theta_{n+k}$ for any $n,k$ and  assume additionally  that there is some  some $c>0$  such that for every $k\in\N$ the following holds
\begin{equation}\label{cond}
\liminf_{n}\frac{\theta_{n+k}}{\theta_{n}}>c. 
\end{equation}
\br Notice that the above property is satisfied for example if $\theta_n^{1/n}\to 1$, as $n\to\infty$. 
\er
Fix a core tree $\mc{R}$ on $\N$. Given any $\mu_j\in\mc{R}$, $j>0$, let 
$$
I_j=\{\beta\in\mc{R}: \ |\beta|=|\mu_j|, \beta>_{lex}\mu_j\}\cup\bigcup_{|\beta|=|\mu_j|, \beta<_{lex}\mu_j}\suc \beta
$$
and $n_j=\# I_j$. 

Using condition \eqref{cond} pick an increasing sequence $(k_r)_r\subset\N$ with $\theta_{r+k}/\theta_k>c$ for any $k\geq k_r$ and $r\in\N$. 

Now consider a block sequence $(x_n)$ such that each $x_n$ has a periodic RIS tree-analysis $(x_\al^n)_{\al\in\mt_n}$ of height $n\in\N$ with core $\mc{R}$. We take the associated norming functionals  $(f_n)$ with a tree-analysis $(f_\al^n)_{\al\in\mt_n}$. Note that the functionals $(f_{n})$ belong also to the norming set of the space $X$.

\bl\label{2o} Let $r\in\N$, $\beta\in\mc{R}$ with $m_{\beta}\geq k_{r+\ord(\beta)}$. Then for any 
$F\in\mc{S}_{r}$ with $F>|\beta|$ we have
\begin{equation}
  \label{eq:2}
 \norm[\sum_{n\in F}\sum_{\al\in\mt_n,\upsilon(\alpha)=\beta}f_{\al}^{n}]\leq c^{-1}.\notag
\end{equation}
\el
\begin{proof}
Recall that each $f_\al^n$, $\upsilon(\al)=\mu_{\beta}$, $n>|\beta|$, is of the form $f_\al^n=\theta_{m_{\beta}}\sum_{\gamma\in\suc(\al)}f_{\gamma}^{n}$, with $(f_{\gamma}^{n})_{\gamma\in\suc(\al)}$ $\mc{S}_{m_{\beta}}$-admissible. On the other hand $(f_{\al}^n)_{ \upsilon(\al)=\mu_{\beta}}$ is $\mc{S}_{\ord(\mu_{\beta})}$-admissible by Remark \ref{rem} (3). Therefore 
$$
g=\theta_{m_{\beta}+\ord(\beta)+r} \sum_{n\in F} \sum_{\al\in\mt_n:\upsilon(\al)=\beta} \sum_{\gamma\in\suc(\al)}f_{\gamma}^n 
$$
is in the norming set of $X_D$ by (D1). Since $m_{\beta}>k_{r+\ord(\beta)}$ it follows that $\frac{\theta_{m_{\beta}}}{\theta_{m_{\beta}+\ord(\beta)+r} } \leq c^{-1}$. Therefore
$$
\sum_{n\in F}\sum_{\al\in\mt_n:\upsilon(\al)=\beta}f_\al^n=
\theta_{m_\beta}\sum_{n\in F}\sum_{\al\in\mt_n:\upsilon(\al)=\beta}\sum_{\gamma\in\suc(\al)}f_\gamma^n
=\frac{\theta_{m_\beta}}{\theta_{m_\beta +\ord(\beta)+r}}g
$$
and the result follows.
\end{proof}

\bl\label{2}  Fix $j,r\in\N$, $j>0$. Assume that for any $\beta\in I_j$ we have $m_\beta\geq k_{r+\ord(\beta)}$. Then for any $F\in\mc{S}_r$ with $F> |\mu_j|$  we have
$$
\norm[\sum_{n\in F}f_n-\sum_{n\in F}\sum_{\upsilon(\al)=\mu_j}c_jf_{\al}^n]\leq \frac{n_j}{c}
$$
\el
\bp
For every  $\beta\in I_j$  Lemma~\ref{2o} yields
\begin{equation}
  \label{eq:5}
\norm[\sum_{n\in F}\sum_{\al\in\mt_n:\upsilon(\al)=\beta}f_\al^n]\leq c^{-1}.  
\end{equation}
As each $f_n$, $n>j$ is of the form 
$$
f_n=\sum_{\beta\in I_j}\sum_{\al\in\mt_n: \upsilon(\al)=\beta}c_\beta f_\al^n+\sum_{\al\in\mt_n:\upsilon(\al)=\mu_j}c_j f^n_{\al},
$$ 
with each $c_\beta\in (0,1)$, using \eqref{eq:5} we get
\begin{align*}
\norm[\sum_{n\in F}f_n-\sum_{n\in F}\sum_{\upsilon(\al)=\mu_j}c_jf_{\al}^n]&=\norm[\sum_{\beta\in I_j}\sum_{n\in F}\sum_{\al\in\mt_n:\upsilon(\al)=\beta}c_\beta f^n_\al]\\
&\leq \sum_{\beta\in I_j}\norm[\sum_{n\in F}\sum_{\al\in \mt_n: \upsilon(\al)=\beta}f_\al^n]\leq \frac{n_j}{c}
\end{align*}
and we finish the proof. 
\ep
\br\label{rem2} 
These Lemmas show the crucial technical difference between $(\mc{A}_n)$ and $(\mc{S}_n)$ cases. In $(\mc{A}_n)$ case we are able to produce a suitable sequence $(x_n)$ with a tree-analysis $(x_\al)_{\al\in\mt_n}$ modeled on some core tree $\mc{R}$ in a such a way, that on a fixed level of each tree $\mt_n$ we have the same number of nodes and the nodes form a RIS. 

To achieve analogous situation in the spaces defined by Schreier families, we have to introduce periodic repetition in the structure of RIS, cf. Remark \ref{rem-z0}. Recall that in the periodic RIS tree-analysis with a fixed core we allow (as it is inevitable in case of Schreier families) the number of nodes of any fixed level of trees $\mt_n$ to increase as $n\to \infty$. However, we keep the same number of weights at each fixed level, i.e. $\{w(f_\al^n): \al\in\mt_n,\ |\al|=l\}=\{m_\beta:\ \beta\in\mc{R}, \ |\beta|=l\}$ for any $n\in\N$ and any fixed level $l$. 

With this property we are able to prove Lemma \ref{2},  crucial for the estimation providing strict singularity of the constructed operator. On the other hand Lemma \ref{2o} ensures that any functionals related to the same node in $\mc{R}$ produces a $c_0$-average. This approach allows us to obtain tight estimates of norms of sums of considered functionals.
\er
With this preparation we are ready to prove the following theorem, whose proof extends the method of Prop. \ref{adt}.
\bt Let a sequence $(\theta_n)\searrow 0$ with $\theta_n\theta_k\leq\theta_{n+k}$, $n,k\in\N$, satisfy the following condition for some $c>0$
$$
\liminf_{n}\frac{\theta_{n+k}}{\theta_{n}}>c\ \ \ \text{ for any }\ \  k\in\N\,.
$$
Then there is a bounded strictly singular non-compact operator $T$ on the space $X_D$ defined by the families $(\mc{S}_n)$. 
\et
\bp We pick a core tree $\mc{R}$ with parameters $(m_j)_{j\geq 0}$, $(M_j)_{j\geq 0}$ and an increasing sequence $(r_j)_{j>0}\subset L$ inductively on $j$ to satisfy (R0), (R1), (R2) and the following 
\begin{enumerate}
 \item[(R3)] $\rho_{r_j}(\sum_{i<j}M_i+\sum_{i<j}m_{i})<\frac{1}{2^j}$, $j\in\N$,
 \item[(R4)] $k_{r_j+\ord(\mu_j)}\leq m_{j}$, $j\in\N$.
\end{enumerate}
We describe the inductive construction: we choose freely $m_0,m_1,r_1$ and $M_0,M_1$ satisfying (R2). Fix $j\in\N$, $j\geq 2$ and assume we have defined $m_i$ with $i<j$. Then we have also $(M_i)_{i<j}$ and 
$\ord(\mu_j)$ and we choose $r_j$ to satisfy (R3) and $m_j$ to satisfy (R0), (R1) and (R4). Then we choose $M_j$ to satisfy (R2).

\

By Lemma \ref{??} there is a block sequence $(x_n)\subset X$ with a tree-analysis with core $\mc{R}$ and the block sequence of associated norming functionals $(f_n)\subset X^*$ so that $i_{p_n}\leq x_n$, where $\supp x_n\in\mc{S}_{p_n}$ for each $n\in\N$ (see Remark \ref{rem} (4)). Thus, with the abuse of notation, by $(\mc{S}_n)$-equivalence of the basis in $X$ and $X_D$, we shall treat vectors $(x_n)$ as elements of $X_D$ and $(f_n)$ as elements of $X_D^*$.

\bcl\label{3} With the above notation the following hold 
\begin{enumerate}
\item[(F1)] the sequence $(f_n)$ is seminormalized in $X_D$,
  \item[(F2)] for any $j\in\N$, any  $F\in\mc{S}_{r_{j+1}}$ with $j+1\leq F$ we have 
$$
\norm[\sum_{n\in F}f_n-\sum_{n\in F}\sum_{\al\in\mt_n:\upsilon(\al)=\mu_j}c_jf^n_\al]_D\leq \frac{n_j}{c}\,,
$$
\item[(F3)] for any $j\in\N$ and $n\in\N$ with $n\geq |\mu_j|+2$ we have
$$
\norm[\sum_{\al\in\mt_n, \upsilon(\al)=\mu_j}f_\al^n]_D\leq \frac{1}{c}\,.
$$
\end{enumerate}
\ecl
To show the Claim apply Lemma~\ref{??} in $X$, obtaining $\norm[f_n]\geq 1/3$ for any $n\in\N$. Therefore by the condition on the supports of $x_n$ and (D2) we obtain $\norm[f_n]_D\geq 1/3C$, $n\in\N$.  
The condition (R4) implies that for any $j>0$ we have $m_\beta\geq k_{r_{j+1}+\ord(\beta)}$ for any $\beta\in I_j$ (as $\min_{lex}I_j=\mu_{j+1}$ and $(k_r)$ increases), therefore (F2) follows by Lemma \ref{2} as  $|\mu_j|\leq j$. To prove (F3) use Lemma \ref{2o} for $r=0$. 

\br 
The above Claim shows why we cannot apply directly Prop. \ref{adt} here; while assumption (1) in Prop. \ref{adt} has its counterpart in property (D3), instead of assumption (2) of Prop. \ref{adt} we have (F2) and (F3), which requires additional estimations. 
\er

We continue the proof of the theorem.

For any $n\in\N$ let $g_n=f_{r_{n+1}}$. We define an operator $Tx=\sum_{n\in\N}g_n(x)e_{t_n}$ on any $x\in X_D$ with finite support, where $t_n=i_{r_{n+1}}$ for each $n\in\N$. We shall show that $T$ is bounded and strictly singular. Notice that $T$ is non-compact, as the sequence $(f_n)$ is seminormalized by (F1). 

We take any $x\in X_D$ with a finite support and a norming functional $f$ with $f(Tx)=\norm[Tx]$. Let for any $j$ 
$$
B_j=\{ n\in\N:\ \rho_{r_{j+1}}<|f(e_{t_n})|\leq \rho_{r_j}\}, \ \ \  D_j=B_j\cap\{1,\dots, \sum_{i<j}M_i+\sum_{i<j}m_i\}, 
$$ 
For any $n\in B_j\setminus D_j$ let $h_n=g_n-c_ju_n$, where $u_n=\sum_{\al\in\mt_n:\upsilon(\al)=\mu_j}f^{r_{n+1}}_\al$. For any $j\in\N$ put
\begin{align*}
C_j^1&=\{n\in B_j\setminus D_j:\ \ h_n(x)\geq 0,\ f(e_{t_n})\geq 0\}\\
C_j^2&=\{n\in B_j\setminus D_j:\ \ h_n(x)< 0,\ f(e_{t_n})\geq 0\}\\
C_j^3&=\{n\in B_j\setminus D_j:\ \ h_n(x)\geq 0,\ f(e_{t_n})<0\}\\
C_j^4&=\{n\in B_j\setminus D_j:\ \ h_n(x)< 0,\  f(e_{t_n})<0\}\\
\end{align*}
Notice that each set $\{i_{r_{n+1}}:\ n\in C_j^k\}\in\mc{S}_{r_{j+1}} $ with $\{i_{r_{n+1}}: n\in  C_j^k\}\geq i_{r_{j+1}}$ for each $j$ by (D3). 

Now compute 
\begin{align}\label{parts}
\norm[Tx]_D&\leq \sum_{j=1}^{j_0-1}|\sum_{n\in B_j}g_n(x)f(e_{t_n})|+\sum_{j=j_0}^\infty|\sum_{n\in D_j}g_n(x)f(e_{t_n})|\\
&+\sum_{j=j_0}^\infty |\sum_{n\in B_j\setminus D_j}h_n(x)f(e_{t_n})|+\norm[\sum_{j=j_0}^\infty c_j\sum_{n\in B_j\setminus D_j}u_n(x)e_{t_n}]_D\,.\notag
\end{align}
Estimate the second term by (R3)
$$
\sum_{j=j_0}^\infty|\sum_{n\in D_j}g_n(x)f(e_{t_n})|\leq  \sum_{j=j_0}^\infty \rho_{r_j}(\sum_{i<j}M_i+\sum_{i<j}m_i)\norm[x]_D\leq \sum_{j=j_0}^\infty\frac{1}{2^j}\norm[x]_D
$$
Estimate the third term by Lemma \ref{2}
\begin{align*}
\sum_{j=j_0}^\infty |\sum_{n\in B_j\setminus D_j}h_n(x)f(e_{t_n})|&\leq \sum_{j=j_0}^\infty\sum_{k=1}^4|\sum_{n\in C_j^k}h_n(x)f(e_{t_n})|\\
& \leq \sum_{j=j_0}^\infty\sum_{k=1}^4|\sum_{n\in C_j^k}h_n(x)|\max_{n\in C_j^k}|f(e_{t_n})| & \text{ by (F2)} \\
& \leq \sum_{j=j_0}^\infty\frac{4}{c}\rho_{r_j} n_j\norm[x]_D\leq \frac{4}{c}\sum_{j=j_0}^\infty\frac{1}{2^j}\norm[x]_D\,.
\end{align*}
where the last inequality follows from (R3) and the fact that $n_j\leq \sum_{i<j}M_i$.

In order to estimate the last part pick $k_0$ such that $\mu_{j_0}\in\suc(\mu_{k_0})$ and compute
\begin{align*}
\norm[\sum_{j=j_0}^\infty c_j\sum_{n\in B_j\setminus D_j}u_n(x)e_{t_n}]_D\leq \sum_{k=k_0}^\infty c_k\theta_{m_k}\norm[\sum_{\mu_j\in\suc(\mu_k)}\sum_{n\in B_j\setminus D_j}u_n(x)e_{t_n}]_D\leq\dots
\end{align*}
Notice that $F_k=\{i_{r_{n+1}}:\ n\in B_j\setminus D_j, \ \mu_j\in\suc(\mu_k)\}\in \mc{S}_{r_{j_k+1}+m_k}$ for each $k\in\N$, where $j_k=\max\{j:\ \mu_j\in\suc(\mu_k)\}$. By definition of $D_j$ we have also $F_k\geq i_{r_{j_{k+1}}+m_k}$. Therefore we continue 
\begin{align*}
\dots&\leq C\sum_{k=k_0}^\infty\theta_{m_k}\norm[\sum_{\mu_j\in\suc(\mu_k)}\sum_{n\in B_j\setminus D_j}u_n(x)e_{t_n}] \quad &  \text{by (D2)}\\
&\leq \frac{C}{c}\sum_{k=k_0}^\infty\theta_{m_k}\norm[x]\quad &  \text{by Fact \ref{fact1} and (F3)}\\
&\leq \frac{C}{c}\sum_{k=k_0}^\infty\frac{1}{2^k}\norm[x]_D \quad & \text{by (D1) and (R1)}. 
\end{align*}
To show boundedness of $T$ take $j_0=1$.  The first term of \eqref{parts} does not appear and the result follows from the above inqualities.

To show strict singularity we have to handle the first term of estimation \eqref{parts}. Consider the norm $\norm_G$, with $G=\{g_n: \ n\in\N\}$ and $\mc{F}=\mc{S}_{r_{j_0+1}}$ defined as in Remark \ref{ps}. As  $\norm_{G}\leq \theta_{r_{j_0+1}}^{-1}\norm_{D}$ by (D1) and $X_{D}$ is reflexive it follows that for any $\e>0$ in any block subspace of $X_D$ there is a vector $x$ of norm 1 such that $\norm[x]_{G}<\e$. 

Assume the vector $x\in X_D$ satisfies $\norm[x]_D=1$ and $\norm[x]_{G}\leq \frac{1}{j_{0}^{2} } $. As in the estimation \eqref{g-norm} we obtain
\begin{align*}
\sum_{j=1}^{j_0-1}|\sum_{n\in B_j}g_n(x)f(e_{t_n})| &\leq \sum_{j=1}^{j_0-1} \sum_{n\in B_j}|g_n(x)|\,|f(e_{t_n})|\leq \frac{1}{j_0}
 \end{align*}
Putting together all the above estimates we obtain
\begin{align*}
\norm[Tx]_D& \leq 
\frac{1}{j_0}+\frac{8}{c2^{j_0-1}}+\frac{C}{c2^{k_0-1}}\,.
\end{align*}
Since we can pick such vector $x$ for any $j_0$ and in any block subspace of $X_D$, and $k_0\to \infty$ as $j_0\to\infty$,  we finish the proof of strict singularity of $T$.
\ep
\br
In the case the space $X_{D}$ is an HI space  we can prove the strictly singularity of $T$ using that its  kernel  is infinite dimensional. Indeed for the complex case it follows by Gowers-Maurey result for the operators in complex  HI spaces \cite{GM}, while for the real case it follows from Argyros-Tolias result \cite{Atol}. 
\er

Now we discuss the case of families $(\mc{A}_n)$. 

\bt Let a sequence $(\theta_n)\searrow 0$ with $\theta_n\theta_k\leq\theta_{nk}$, $n,k\in\N$, satisfy $\theta_nn^a\to \infty$ for any $a>0$. 

Then there is a bounded strictly singular non-compact operator $T$ on the $X_D$ defined by the families $(\mc{A}_n)$. 
\et

The proof goes analogously and is much simplified, as we are able to use only usual averages, without taking periodic repetition of the RIS structure, cf. Remark \ref{rem2}. For this reason we shall not repeat the reasoning, but indicate the technical differences. 

Notice first the following
\begin{fact} Let the sequence $(\theta_n)$ satisfy the assumption of the Theorem. Then for any $k\in\N$ we have
 $$
\liminf_n\frac{\theta_{nk}}{\theta_n}=1
$$
\end{fact}
To prove the Fact notice that with the assumption on $(\theta_n)$ can write $\theta_n$ as 
$\theta_n=1/n^{1/q_n}$, $n\in\N$, with $q_n\nearrow \infty$. Thus 
$$
\frac{\theta_{nk}}{\theta_n}=\frac{n^{1/q_{n}}}{(nk)^{1/q_{nk}}}=\frac{n^{1/q_{n}-1/q_{nk}}}{k^{1/q_{nk}}}\geq \frac{1}{k^{1/q_{nk}}}\to 1,  \text{ as }n\to\infty
$$
We take now a core tree $\mc{R}$ with parameters $(m_j)$ and $(M_j)=(m_j)$ and an increasing sequence $(r_j)$ satisfying suitably modified conditions for any $j\in\N$: 
\begin{enumerate}
\item[(R0)] $m_{j+1}\geq w(\theta_{m_j}^3/4,q_j)$ with some $q_j\in\N$ satisfying $\theta_{q_j}\leq \theta_{m_j}^4$ for any $j\in\N$,
\item[(R1)] $\theta_{m_j}<\frac{1}{2^{j+1}}$,
\item[(R2)] $m_j>4(1-\theta_{1})^{-1}\theta_{m_j}^{-4}$,
 \item[(R3)] $ \theta_{r_j}(\sum_{i<j}m_{i})< \frac{1}{2^j}$,
\item[(R4)]  $k_{r_j}\leq m_{j} $.
\end{enumerate}
We describe the inductive construction. As by the assumption $\theta_nn^{1/4}\to 0$ picking $(m_j)$ at each step large enough we ensure (R2).  We choose freely $m_0,m_1,r_1$. Fix $j\in\N$, $j\geq 2$ and assume we have defined $m_i$ with $i<j$. Then we have $\ord(\mu_j)$ and we choose $r_j$ to satisfy (R3) and $m_j$ to satisfy(R1) and  (R4). 

We take a block sequence $(x_n)$ with a RIS tree-analysis $(x_\al^n)_{\al\in\mt_n}$ of the simplest form: for any $\al\in\mt_n$ we have $N_\al=1$. According to the Remark \ref{rem} (2) we can regard each tree $\mt_n$ as $\mc{R}\cap\cup_{i=0}^n\N^i$. 

Lemma~\ref{2} takes a simpler form in such a situation. We pick an increasing sequence $(k_r)_r\subset\N$ such that $\theta_{rk}/\theta_k>1/2$ for any $k\geq k_r$ and any $r\in\N$. Then given $j,r\in\N$, $j>0$ with $m_\beta\geq k_r$ for any $\beta\in I_j$ we have  for any $F\subset\N$ with $F\geq j$, $\# F\leq r$ the following
$$
\norm[\sum_{n\in F}f_n-\sum_{n\in F}c_{\mu_j}^nf_{\mu_j}^n]\leq 2n_j\,.
$$
The rest of the reasoning in proof of the theorem above follows straightforward. 
\begin{remark}
  In a recent work S.A. Argyros and P. Motakis, \cite{amot}, present an example of a reflexive  HI Banach space $X_{ISP}$, built on the Tsirelson space, with the property that for every three strictly singular operators $T_{1},T_{2},T_{3}$ on $X_{ISP}$ the composition operator  $T_{1}T_{2}T_{3}$ is compact. This yields  that every strictly singular operator on $X_{ISP}$ has an invariant subspace. The above result leads to the following question.

Let  $X$ be an HI Banach space  built on a mixed Tsirelson space defined by the Schreier families $(\mc{S}_{n})$. Does there exists $n\in\N$ such that for every strictly singular  operators  $T_1,\dots,T_n$ the operator $T_1\dots T_n$ is compact?

\end{remark}

\end{document}